\documentclass[12pt]{amsart}
\usepackage{amssymb}

\usepackage[T1]{fontenc}
\usepackage[cp1250]{inputenc}
\usepackage{amssymb}
\usepackage{amsmath}
\usepackage{amsfonts}
\usepackage{amsthm}
\usepackage{graphics}
\usepackage{color}

\newcommand {\RR}  {\mathbb{R}}
\newcommand {\QQ}  {\mathbb{Q}}

\newtheorem{thm}{Theorem}[section]
\newtheorem{cor}[thm]{Corollary}

\newtheorem{rem}[thm]{Remark}

\newtheorem{conj}[thm]{Conjecture}

\newtheorem{exam}[thm]{Example}

\begin{document}

\title{Pairs of Pythagorean triangles with given catheti ratios}
\author{M. Ska\l ba \and  M. Ulas }
\begin{abstract}
In this note we investigate the problem of finding pairs of Pythagorean triangles $(a, b, c), (A, B, C)$, with given catheti ratios $A/a, B/b$. In particular, we prove that there are infinitely many essentially different ("non-similar") pairs of Pythagorean triangles  $(a, b, c), (A, B, C)$ satisfying given proportions, provided that $Aa\neq Bb$.
\end{abstract}

\address{Mariusz Ska{\l}ba, Institute of Mathematics\\
University of Warsaw\\
Banacha 2\\
02-097 Warszawa, Poland}
\email{skalba@mimuw.edu.pl}

\address{Maciej Ulas, Institute of Mathematics\\
Jagiellonian University  \\
\L ojasiewicza 6 \\
30-348 Krak\'{o}w, Poland}
\email{maciej.ulas@uj.edu.pl}



\maketitle

\section{Introduction}\label{sec1}
In \cite{ML} the author investigated the following problem stated by Leech (as was remarked by Smyth in \cite{S}):
\\
\noindent {\it Find two rational right-angled triangles on the same base whose heights are in the ratio $n:1$ for $n$ an integer greater than 1.}
\\
By considering an equivalent elliptic curve problem, he find parametric solutions for certain values of $n$ and present results of a numerical search for solutions with $n\in\{1,\ldots,999\}$. Motivated by his findings we deal with the general problem of finding pairs of Pythagorean triangles with given catheti ratios. More precisely, throughout the paper $(a,b,c)$ will denote a Pythagorean triple, i.e. a triple of positive integers satisfying the condition
$$
a^2+b^2=c^2\,\mbox{ and } a,b,c\in\mathbb N=\{1,2,\ldots\}.
$$
Given two such triples $(a,b,c)$ and $(A,B,C)$ we are interested whether there exist essentially different pairs $(a',b',c')$, $(A',B',C')$ satisfying
\begin{equation}\label{main}
\frac{A'}{a'}=\frac{A}{a},\ \ \frac{B'}{b'}=\frac{B}{b}.
\end{equation}
As we will see in the sequel under some mild conditions on the pair of triples $(a,b,c), (A,B,C)$, the problem has a positive answer. More precisely, in Section \ref{sec2} we prove that there are infinitely many essentially different pairs $(a',b',c'), (A',B',C')$ solving (\ref{main}) provided that $Aa\neq Bb$. In Section \ref{sec3} we investigate the case $Aa=Bb$ and prove that there are infinitely many pairs $(a,b,c), (A,B,C)$ such that the system (\ref{main}) has infinitely many solutions $(a',b',c'), (A',B',C')$. We also present results of our numerical search and observe that for certain pairs $(a,b,c), (A,B,C)$ there is no non-trivial solution. Finally, in the last section we investigate the general problem of finding pairs of Pythagorean triangles $(a, b, c), (A, B, C)$ such that $A/a=\mu, B/b=\nu$, where $\mu, \nu$ are given rational numbers. We reduce the problem to the case when $\mu=1$ and prove that the set of those $\nu$ is dense in the Euclidean topology in $\QQ^{+}$.



\section{The case $Aa\neq Bb$}\label{sec2}

In this section we are interested in finding the solutions of system (\ref{main}) under the assumption $Aa\neq Bb$. More precisely, together with given triples $(a,b,c)$ and $(A,B,C)$ we will consider the following elliptic curve
\begin{equation}
H: v^2=u(a^2u+b^2)(A^2u+B^2)
\end{equation}
The basic observation is that each point $Q=(u,v)\in H(\mathbb Q)$ of the form $2P$ ($P\in H(\mathbb Q)$) gives the desired pair of triangles $(a',b',c')$, $(A',B',C')$.
Namely, if $(u,v)=P+P$ then $u=s^2$, $a^2u+b^2=t^2$ and $A^2u+B^2=r^2$ with some $s,t,r\in\mathbb Q$. Hence
$$
\left\{
\begin{array}{ccccc}
(as)^2 & + & b^2 & = & t^2\\
(As)^2 & + & B^2 & = & r^2
\end{array}
\right.
$$
and $(a',b',c')=(as,b,t)$, $(A',B',C')=(As,B,r)$ do work. The point $Q'=(1,cC)$ corresponds to the initial situation $(a',b',c')=(a,b,c)$, $(A',B',C')=(A,B,C)$. Hence the natural question in this context is whether the point $Q'$ is of infinite order or (at least) is the rank of $H$ positive? We will investigate both problems in the sequel.

After change of variables $x=a^2A^2u$, $y=a^2A^2 v$ we get a standard form
\begin{equation}\label{Ecurve}
E: y^2=x(x+A^2b^2)(x+a^2B^2),
\end{equation}
and the image of the point $Q'$ takes the form $Q^{\ast}=(a^2A^2,a^2A^2cC)$. The group $E(\mathbb{Q})$ contains exactly 3 points of order 2, hence
the torsion part $\operatorname{Tor}(E(\mathbb{Q}))$ cannot be cyclic. By the Mazur theorem $\operatorname{Tor}(E(\mathbb{Q}))=\mathbb{Z}_{2}\times\mathbb{Z}_{2n}$ for some $n\in\{1,2,3,4\}$. If $P+P=(0,0)$ then $|P|=4$, hence $n=2$ or $n=4$. It follows that if $Q^{\ast}$ has a finite order then $|Q^{\ast}|=4$ and $|2Q^{\ast}|=2$. By duplication formula
$$
x(2Q^{\ast})=\frac{(a^2A^2-B^2b^2)^2}{4c^2C^2}
$$
we infer that $|2Q^{\ast}|=2$ if and only if $Aa=Bb$. Geometrically this means that $(a,b,c)$ and $(A,B,C)$  are " skew-similar". In other cases the point $Q^{\ast}$ has infinite order. In this way we have proved

\begin{thm}
If $Aa\neq Bb$ then there exist infinitely many essentially different ("non-similar") pairs of Pythagorean triangles  $(a',b',c')$, $(A',B',C')$ satisfying given proportions {\rm (\ref{main})}.
\end{thm}

\section{The case $Aa=Bb$}\label{sec3}

In previous section we have proved that provided $Aa\neq Bb$, then there are infinitely many Pythagorean triangles $(a',b',c')$, $(A',B',C')$ satisfying given proportions {\rm (\ref{main})}.

What is going on in the case $Aa=Bb$? First of all, let us note that primitivity of triangles $(a, b, c), (A, B, C)$ immediately implies that $A=b, B=a$. Then $C=c$ and we deal with the case of cross-similarity of pairs of Pythagorean triangles. In other words, we are interested in the existence of rational points on the curve
$$
E_{a,b}:\;y^2=x(x+a^4)(x+b^4).
$$
We prove \begin{thm}
There are infinitely many primitive Pythagorean triples $(a,b,c)$ such that the elliptic curve
$$
E_{a,b}:\;y^2=x(x+a^4)(x+b^4)
$$
has positive rank.
\end{thm}
\begin{proof}
Let $a=2uv, b=u^2-v^2, c=u^2+v^2$ be a Pythagorean triple and write $f(x)=x(x+a^4)(x+b^4)$. We note the crucial identity
$$
f(-a^{3}b)=\left(a^3b(b-a)\right)^2(u^4+2u^3v+2u^2v^2-2uv^3+v^4).
$$
In other words $f(-a^3b)$ is a square if and only if there is a rational point on the quartic curve
$$
C:\;W^2=U^4+2U^3+2U^2-2U+1.
$$
The curve $C$ is birationally equivalent with the elliptic curve
$$
E':\;Y^2=(X-24)(X-6)(X+30),
$$
via the mapping
$$
\phi:\;C\ni (U,W)\mapsto (X,Y)\in E',
$$
where
$$
(X,Y)=(6(3U^2+3U+1-3W), 54(2U^3+3U^2+2U-1-(1+2U)W)),
$$
and the inverse $\phi^{-1}:\;E'\ni (X,Y)\mapsto (U,W)\in C$ is given by
$$
(U,W)=\left(\frac{72+Y-3X}{6(X+3)},\;\frac{Y^2+162Y+6588-2X^3-9X^2}{36(X+3)^2}\right).
$$
A quick computation with {\sc Magma} \cite{Mag}, reveals that rank of $E'(\QQ)$ is equal to 1. One can check that the generator of the infinite part is $P=(42,-216)$.

Now we compute multiplies $nP=(X_{n},Y_{n})$ for $n=1,2,\ldots$ and then consider corresponding points $(U_{n}, W_{n})=\phi^{-1}(nP)$. By writing $U_{n}=u_{n}/v_{n}, W_{n}=w_{n}/v_{n}^2$ with $\gcd(u_{n}w_{n},v_{n})=1$ we get the triple $(a_{n},b_{n},c_{n})=(2u_{n}v_{n},|u_{n}^2-v_{n}^2|,u_{n}^2+v_{n}^2)$ which leads to the curve $E_{a_{n},b_{n}}$ with point of infinite order
$$
(x_{n},y_{n})=(-a_{n}^3b_{n},4a_{n}b_{n}(b_{n}-a_{n})w_{n}).
$$

In order to finish the proof we need to know that for a given pair $(a_{n},b_{n})$ there are only finitely many pairs $(a_{m}, b_{m})$ with $m>n$ such that the curves $E_{a_{n},b_{n}}$, $E_{a_{m},b_{m}}$ are isomorphic. This can be seen as follows: by standard properties of elliptic curves we know that the curves $E_{a,b}, E_{A,B}$ are isomorphic if and only there is a linear isomorphism $\psi:\;E_{a,b}:\;(x,y)\mapsto (p^2x+q,p^3y)\in E_{A,B}$ for some $p, q\in\mathbb{Q}$. Let $f_{a,b}(x,y)$ be the polynomial defining the curve $E_{a,b}$. We see that $E_{a,b}\simeq E_{A,B}$ if and only if
$$
f_{A,B}(p^2x+q,p^3y)=Cf_{a,b}(x,y),
$$
for some $C\in\mathbb{Q}$. By comparing the coefficients on both sides we get that $C=p^6$ and need to consider the following system of equations:
\begin{equation*}
\begin{cases}
\begin{array}{lll}
  q(A^4+q)(B^4+q) & = & 0 \\
  p^2(A^4B^4-a^4b^4p^4+2(A^4+B^4)q+3q^2) & = & 0\\
  p^4(A^4+B^4-(a^4+b^4)p^2+3q) & = & 0
\end{array}
\end{cases}.
\end{equation*}
Let $G$ be the Gr\"{o}bner basis of the polynomials defining the above system in the ring $\mathbb{Q}[a,b,A,B][p,q]$. In other words we treat $a, b, A, B$ as constants. We have $|G|=14$ and note that $G\cap \mathbb{Q}[a,b,A,B]=\{F\}$, where
\begin{align*}
F=&(A b-a B)(A b + a B)(a A - b B)(a A + b B)\times\\
  &(A^2 b^2+a^2 B^2)(a^2 A^2 + b^2 B^2)\times \\
  &(a^4 A^4 - A^4 b^4 - a^4 B^4)(A^4 b^4 + a^4 B^4 - b^4 B^4)\times \\
  &(a^4 A^4 - A^4 b^4 + b^4 B^4) (a^4 A^4 - a^4 B^4 + b^4 B^4).
\end{align*}
Assuming that $a, b, A, B$ are positive we see that the only possibility for vanishing of $F(a,b,A,B)$ is the condition $(Ab-aB)(aA-bB)=0$. In other words, if $a_{n}, b_{n}$ and $a_{m}, b_{m}$ with $n\neq m$ come from our construction we need to have $a_{n}a_{m}=b_{n}b_{m}$ or $a_{n}b_{m}=a_{m}b_{n}$. However, let us note that for given $n\in\mathbb{N}$ the system of equations
$$
2uva_{n}=|u^2-v^2|b_{n},\quad w^2=u^4+2u^3v+2u^2v^2-2uv^3+v^4
$$
has only finitely many solutions in integers $u, v$. This means that there are only finitely many values of $m>n$ such that $a_{n}a_{m}=b_{n}b_{m}$. Similar reasoning works in the case of the equality $a_{n}b_{m}=a_{m}b_{n}$. Thus, in both cases, for only finitely many values of $m>n$ we have that the curves $E_{a_{n},b_{n}}, E_{a_{m},b_{m}}$ are isomorphic. In consequence we can construct an infinite set $\mathcal{A}\subset \mathbb{N}$ such that for each $s_{1}, s_{2}\in \mathcal{A}, s_{1}\neq s_{2}$ we have non-isomorphic curves corresponding to $s_{1}, s_{2}$.

\end{proof}

\begin{exam}
{\rm Let $P=(42,-216)$ be the generator of the free part of $E'(\QQ)$, where $E'$ is constructed in the proof of the theorem above. We have $2P=(105/4, 405/8)$ and
$$
\phi^{-1}(2P)=(1/4, -13/16),
$$
i.e., $u_{2}=1, v_{2}=4$, with $2u_{2}v_{2}=8, u_{2}^2-v_{2}^2=-15$. The corresponding triple is $(a_{2}, b_{2}, c_{2})=(8,15,17)$. On the curve $E_{8,15}$ we get the point $(8^{3}\cdot 15, 2296320)$ of infinite order.

In case of $3P=(10698/49, 1097928/343)$ we get
$$
\phi^{-1}(3P)=(69/35, -7274/1225)
$$
and the corresponding triple is $(a_{3}, b_{3}, c_{3})=(4830, 3536,5986)$. The point of infinite order on $E_{a_{3},b_{3}}$ is $(-4830^3\cdot 3536, 3750258651849283392000)$. }
\end{exam}
We believe that the following is true.

\begin{conj}
There are infinitely many primitive Pythagorean triples $(a,b,c)$ such that the elliptic curve
$$
E_{a,b}:\;y^2=x(x+a^4)(x+b^4)
$$
has rank zero.
\end{conj}

\begin{rem}{\rm  We performed a small numerical search for primitive triples $(a,b,c)$ such that the rank of $E_{a,b}$ is zero. We used the {\sc MAGMA} computational package and checked that in the range $a<b<10^4$ (in this range there are exactly 890 primitive Pythagorean triples) there are at least 45 triples $(a,b,c)$ for which the rank of $E_{a,b}$ is equal to zero. In order  to get the data we were interested in, first we generated all triples $(a,b,c)$ in the considered range and dealt with the elliptic curve $E_{a,b}$. Next, we used the procedure {\tt RankBound(E)}, implemented in {\sc Magma}, which gives an upper bound for the rank of the elliptic curve $E$. If the computed bound were equal to 0, we printed the triple $(a,b,c)$. In the table below we collect the data we obtained. However, it is very likely that in the range $a<b<10^4$ there are many more triples $(a,b,c)$ such that the rank of $E_{a,b}$ is 0.

\begin{equation*}
\begin{array}{|l|l|l|}
  \hline
  (a,b,c)           & (a,b,c)           & (a,b,c) \\
 \hline
 (3,4,5)	        &(581,3420,3469)	&(2380,4611,5189) \\
 (5,12,13)	        &(588,2365,2437)	&(2436,2923,3805) \\
 (12,35,37)	        &(612,1075,1237)	&(2725,5628,6253) \\
 (13,84,85)	        &(660,2989,3061)	&(3267,6956,7685) \\
 (19,180,181)	        &(685,9372,9397)	&(3612,6955,7837) \\
 (24,143,145)	        &(780,6059,6109)	&(3751,6840,7801) \\
 (33,56,65)	        &(913,3384,3505)	&(4180,8541,9509) \\
 (64,1023,1025)	        &(924,5893,5965)	&(4251,5180,6701) \\
 (69,2380,2381)  	&(949,2580,2749)	&(4469,5100,6781) \\
 (115,252,277)	        &(1403,1596,2125)	&(4740,5341,7141) \\
 (180,299,349)	        &(1507,9324,9445)	&(5365,9828,11197) \\
 (319,360,481)	        &(1820,8181,8381)	&(5633,7656,9505) \\
 (339,6380,6389)	&(2059,2100,2941)	&(6125,6612,9013) \\
 (473,864,985)	        &(2147,6204,6565)	&(6204,7747,9925) \\
 (540,629,829)	        &(2299,7140,7501)	&(6811,8460,10861) \\ \hline
\end{array}
\end{equation*}
\begin{center}Table. Primitive Pythagorean triples $(a,b,c)$ such that the curve $E_{a,b}$ has rank 0. \end{center}

}
\end{rem}

\section{A general question}\label{sec4}
Now we ask a more difficult question: Does there exist a pair of Pythagoran triangles $(a,b,c)$ and $(A,B,C)$ satisfying
\begin{equation}
\label{general}
\frac{A}{a}=\mu,\ \ \frac{B}{b}=\nu
\end{equation}
where $\mu,\nu$ are given positive rationals. If at least one Pythagorean triple can be non-primitive it is easy to see that the answer depends only on the ratio $\nu/\mu$. Indeed, if triples $(a,b,c), (A,B,C)$ solve the equation (\ref{general}) then the triples $(\mu a, \mu b, \mu c), (A,B,C)$ solve the equation (\ref{general}) with $\mu=1$ and $\nu$ replaced by $\nu/\mu$. Therefore in the sequel we confine ourselves to the pairs with $\mu=1$.

However, before we concentrate on the system (\ref{general}) we state an easier question concerning characterization of all possible pairs $r_{1}, r_{2}$ of positive rational numbers such that there are Pythagorean triangles $(a,b,c)$ and $(A,B,C)$ satisfying
\begin{equation}
\label{general1}
\frac{A}{a}=r_{0},\ \ \frac{B}{b}=r_{1},\ \ \frac{C}{c}=r_{2}.
\end{equation}
As we already noted, without loss of generality, we can assume that $r_{0}=1$. Thus, our problem is equivalent with characterization of solutions of the system
\begin{equation*}
a^2+b^2=c^2,\quad a^2+r_{1}^2b^2=r_{2}^2c^2.
\end{equation*}
By writing $a=2uv, b=u^2-v^2, c=u^2+v^2$ we are interested in solutions of the Diophantine equation
\begin{equation}
(u^2-v^2)^2r_{1}^2-(u^2+v^2)^2r_2^{2}=-(2uv)^2
\end{equation}
with respect to $r_{1}, r_{2}$. However, this is simple because our equation is quadratic in $r_{1}, r_{2}$ and we have the rational point at infinity $(u^2+v^2:u^2-v^2:0)$. In consequence, we obtain the parametrization in the following form
$$
r_{1}=\frac{((u^4-v^4)w+2uv)((v^4-u^4)w+2uv)}{2(u^2-v^2)^2(u^2+v^2)w},\; r_{2}=\frac{(u^4-v^4)^2w^2+4u^2v^2}{2(u^2-v^2)(u^2+v^2)^2w}.
$$
A quick computation reveals that our ratios $r_{1}, r_{2}$ are positive (with fixed values of $u, v$) if and only if the condition $0<w<\frac{2uv}{u^4-v^4}$ is satisfied.
\bigskip

Let us return now to our initial question. Of course finding the solutions of the system (\ref{general}) is more difficult. Indeed, for a concrete $\nu$ we are confronted with the problem of positivity of the rank of the elliptic curve
$$
E_{\nu}:y^2=x(x+1)(x+\nu^2)
$$
(compare the paper \cite{ML} for the case $\nu\in\mathbb N$). We prove a general but weaker theorem.

\begin{thm}\label{all-classes}
For each $\bar{\nu}\in\mathbb Q^{+}$ there exist $s\in\mathbb Q^{+}$ such that for $\nu=\bar{\nu}s^2$  there exist $(a,b,c)$ and $(A,B,C)$ satisfying $A=a$, $B=\nu b$.
\end{thm}
\begin{proof} In order to get the result we consider the following Pythagorean triples
$$
a=2x,\ b=x^2-1,\  c=x^2+1, A=2x,\ B=\frac{(y^2-1)x}{y},\ C=\frac{(y^2+1)x}{y}.
$$
Thus, we consider the equation
$$
\frac{B}{b}=\frac{(y^2-1)x}{(x^2-1)y}=\bar{\nu}s^2.
$$
Equivalently, putting $S=\frac{x^2-1}{x}s$, we are interested in the surface
$$
H:\;\bar{\nu}xyS^2=(x^2-1)(y^2-1)
$$
defined over the rational function field $\QQ(\bar{\nu})$. First we construct a rational curve lying on $H$. We are looking for a rational curve of the form
$$
x=T+1,\quad y=pT+1,\quad S=tT,
$$
where $p, T$ need to be determined. A quick computation with the form of our substitution and the equation defining $H$, reveals that
$$
p=\frac{1}{4}vt^2, \quad T=-\frac{2(\bar{\nu}t^2+4)}{3\bar{\nu}t^2}
$$
does the job. In consequence, we get that the curve
$$
L:\;\begin{cases}
x=\frac{\bar{\nu}t^2-8}{3\bar{\nu}t^2}=:f(t),\\
y=\frac{2-\bar{\nu}t^2}{6},\\
S=-\frac{2 \left(t^2 v+4\right)}{3 t v}.
\end{cases}
$$
lies on the surface $H$. We show that there are infinitely many rational curves lying on $H$. In order to do this we treat the surface $H$ as a genus one curve, say $\mathcal{E}$, defined over the rational function field $\QQ(\bar{\nu},t)$ in the $(y,S)$ plane. Here we use the base change $x=f(t)$. The curve $\mathcal{E}$ is written in the Weierstrass form
$$
\mathcal{E}:\;Y^2=X(X-\bar{\nu}(f(t)^3-f(t)))(X+\bar{\nu}(f(t)^3-f(t))),
$$
where the map $\psi:\; H\ni (y,S)\mapsto (X,Y)\in\mathcal{E}$ is given by
$$
X=vf(t)(f(t)^2-1)y,\quad Y=Sv^2f(t)^2(f(t)^2-1)y.
$$
On the curve $\mathcal{E}$ we have the point $P$ coming from the rational curve $L$ lying on $H$. More precisely, $P=(X,Y)$, where
\begin{align*}
X&=\frac{4 \left(\bar{\nu}t^2-8\right) \left(\bar{\nu}t^2-2\right)^2 \left(\bar{\nu}t^2+4\right)}{81 \bar{\nu}^2t^6},\\
Y&=-\frac{8 \left(\bar{\nu}t^2-8\right)^2 \left(\bar{\nu}t^2-2\right)^2 \left(\bar{\nu}t^2+4\right)^2}{729 \bar{\nu}^3t^9}.
\end{align*}
Because for any given $\bar{\nu}$ positive rational we have $X, Y\in\QQ(t)\setminus\QQ[t]$ we immediately get (via the analogue of the Nagell-Lutz theorem for a rational function field) that the point $P$ is of infinite order on $\mathcal{E}$. This implies the existence of infinitely many rational curves which lie on $H$ (coming from the points $\psi^{-1}(mP)$, where $m=1,2,\ldots $), with fixed $x=f(t)$. Finally, we see that, for any given $\bar{\nu}$ there are infinitely many parametric families of pairs of Pythagorean triples $(a,b,c), (A,B,C)$ satisfying $a=A, B=\nu b=\bar{\nu}s^2 b$ and hence the result.

\end{proof}

\begin{cor}
The set
$$
\Gamma=\{\nu\in\QQ_{>0}:\;\mbox{the elliptic curve}\;E_{\nu}\;\mbox{has positive rank}\}
$$
is dense in the Euclidean topology in the set $\RR_{>0}$.
\end{cor}
\begin{proof}
In order to get the result we could use the parametric curve constructed in the proof of Theorem \ref{all-classes}. However, instead we will treat the equation $y^2=x(x+1)(x+\nu^{2})$ defining the curve $E_{\nu}$ as an equation in the $(y,\nu)$ plane, i.e.,
$$
y^2=x(x+1)\nu^{2}+x^2(x+1),
$$
which is a genus 0 curve defined over rational function field $\QQ(x)$. After rational base change $x=t^2-1$ we see that the curve
$$
y^2=(t^2-1)t^2\nu^{2}+(t(t^2-1))^2
$$
has $\QQ(t)$-rational point $(\nu,y)=(0,t(t^2-1))$. Thus we obtain the parametric solution in the following form
$$
\nu=\frac{2 t \left(t^2-1\right) u}{t^2 \left(t^2-1\right)-u^2},\quad y=\frac{t(t^2-1)((t^2-1)t^2+u^2)}{t^2(t^2-1)-u^2}.
$$
We thus see that, for given $t\in\QQ$, on the elliptic curve $E_{\nu(u)}$, treated over rational function field $\QQ(u)$, we have the point
$P=(t^2-1,y(u))$. It is clear that the point $P$ is of infinite order on $E_{\nu(u)}$.

To finish the proof, note that for any given positive $t\in\QQ$ satisfying the property $t^2(t^2-1)>0$, the rational map $\nu:\RR\ni u\mapsto \frac{2 t \left(t^2-1\right) u}{t^2 \left(t^2-1\right)-u^2}\in\RR$ is continuous
and has the obvious property
 \begin{equation*}
\nu(0)=0,\quad\quad \lim_{u\rightarrow u_{0}}\nu(u)=+\infty,
\end{equation*}
where $u_{0}=\sqrt{t^2(t^2-1)}$. The density of $\QQ$ in $\RR$ together with the above property of the map $\nu=\nu(u)$, immediately implies that the set $\nu(\QQ)\cap \RR_{>0}$ is dense in $\RR_{>}$ in the Euclidean topology. The theorem follows.

\end{proof}

\begin{rem}{\rm  Parameters $\nu$ for which the rank of $E_{\nu}$ is positive arise from a geometric problem, similar to that of congruent numbers. But our Theorem \ref{all-classes} illustrates a difference. Whereas beeing {\it congruent} depends only on the square-class of a considered number the situation with numbers $\nu$ is completely different: by our theorem each square-class contains relevant numbers $\nu$! }
\end{rem}

\bigskip

\noindent {\bf Acknowledgments.} We are grateful to the anonymous referee for careful reading of the manuscript and useful remarks which have improved the presentation.

\end{document}